\documentclass[11pt]{amsart}
\usepackage[colorlinks=true, pdfstartview=FitV, linkcolor=blue, 
citecolor=blue]{hyperref}

\usepackage{amssymb,bbm,amscd}
\usepackage{graphicx}
\usepackage{a4wide}

\theoremstyle{plain}
\newtheorem{theorem}{Theorem}[section]
\newtheorem{prop}[theorem]{Proposition}
\newtheorem{lemma}[theorem]{Lemma}
\newtheorem{coro}[theorem]{Corollary}
\newtheorem{fact}[theorem]{Fact}

\theoremstyle{definition}
\newtheorem{remark}[theorem]{Remark}

\numberwithin{equation}{section}
 
\newcommand{\dd}{\,\mathrm{d}}
\newcommand{\ii}{\ts\mathrm{i}\ts}
\newcommand{\ee}{\,\mathrm{e}}
\newcommand{\ts}{\hspace{0.5pt}}
\newcommand{\nts}{\hspace{-0.5pt}}

\DeclareMathOperator{\dens}{\mathrm{dens}}
\DeclareMathOperator{\card}{\mathrm{card}}

\newcommand{\fm}{\mathfrak{m}}
\newcommand{\fM}{\mathfrak{M}}

\newcommand{\vL}{\varLambda}
\newcommand{\vU}{\varUpsilon}
\newcommand{\cA}{\mathcal{A}}
\newcommand{\cB}{\mathcal{B}}
\newcommand{\cD}{\mathcal{D}}
\newcommand{\cE}{\mathcal{E}}
\newcommand{\cI}{\mathcal{I}}

\newcommand{\cO}{\mathcal{O}}
\newcommand{\cT}{\mathcal{T}}
\newcommand{\ZZ}{\mathbb{Z}\ts}
\newcommand{\RR}{\mathbb{R}\ts}
\newcommand{\CC}{\mathbb{C}\ts}

\newcommand{\MM}{\mathbb{M}}
\newcommand{\NN}{\mathbb{N}}
\newcommand{\TT}{\mathbb{T}}
\newcommand{\XX}{\mathbb{X}}
\newcommand{\YY}{\mathbb{Y}}

\newcommand{\exend}{\hfill $\Diamond$}

\newcommand{\Mat}{\mathrm{Mat}}

\newcommand{\myfrac}[2]{\frac{\raisebox{-2pt}{$#1$}}
      {\raisebox{0.5pt}{$#2$}}}

\begin{document}

\title[A family of binary inflation rules]{Spectral analysis
of a family of binary inflation rules }

\author{Michael Baake} 
\address{Fakult\"{a}t f\"{u}r Mathematik,
  Universit\"{a}t Bielefeld,\newline \hspace*{\parindent}Postfach
  100131, 33501 Bielefeld, Germany}
\email{$\{$mbaake,cmanibo$\}$@math.uni-bielefeld.de}

\author{Uwe Grimm} 
\address{School of Mathematics and Statistics,
  The Open University,\newline \hspace*{\parindent}Walton Hall, 
  Milton Keynes MK7 6AA, United Kingdom} 
\email{uwe.grimm@open.ac.uk}

\author{Neil Ma\~{n}ibo}

\begin{abstract}
  The family of primitive binary substitutions defined by
  $1 \mapsto 0 \mapsto 0 1^m$ with \mbox{$m\in\NN$} is investigated.
  The spectral type of the corresponding diffraction measure is
  analysed for its geometric realisation with prototiles (intervals)
  of natural length. Apart from the well-known Fibonacci inflation
  ($m=1$), the inflation rules either have integer inflation factors,
  but non-constant length, or are of non-Pisot type. We show that
  all of them have singular diffraction, either of pure point type
  or essentially singular continuous.
\end{abstract}

\maketitle
\thispagestyle{empty}

\section{Introduction}

Due to the general interest in substitutions with a multiplier that
is a Pisot--Vijayaraghavan (PV) number, and the renewed interest in
substitutions of constant length, other cases and classes have been a
bit neglected. In particular, the analysis of non-PV inflations is
clearly incomplete, although they should provide valuable insight into
systems with singular continuous spectrum. This was highlighted in a
recent example \cite{BFGR}, where the absence of absolutely continuous
diffraction could be shown via estimates of certain Lyapunov
exponents. The same method can also be used for substitutions of
constant length \cite{BG15,Neil} to re-derive results that are
known from \cite{Q,Bart} in an independent way.

Here, we extend these methods to an entire family of binary inflation
rules, namely those derived from the substitutions
$1 \mapsto 0 \mapsto 0 1^m$ with $m\in\NN$ by using prototiles of
natural length.  The inflations are not of constant length, and all
have singular spectrum (either pure point or mainly singular
continuous), as previously announced in \cite{BG-conf}. More
precisely, we prove the following result, the concepts and
details of which are explained as we go along.

\begin{theorem}\label{thm:main-goal}
  Consider the primitive, binary inflation rule\/
  $1 \mapsto 0 \mapsto 01^m$ with\/ $m\in\NN$, and let\/
  $\widehat{\gamma^{}_{u}}$ be the diffraction measure of the
  corresponding Delone dynamical system that emerges from the left
  endpoints of the tilings with two intervals of natural length,
  where\/ $u = (u^{}_{0}, u^{}_{1})$ with\/ $u^{}_{0} u^{}_{1} \ne 0$
  are arbitrary complex weights for the two 
  types of points. Then, one has the following three cases.
\begin{enumerate}\itemsep=2pt
  \item For\/ $m=1$, this is the well-known Fibonacci chain, which
    has pure point diffraction and, equivalently, pure point dynamical
    spectrum.
  \item When\/ $m=\ell (\ell+1)$ with\/ $\ell\in\NN$, the inflation
    multiplier is an integer, and the diffraction measure as well
    as the dynamical spectrum is once again pure point.
  \item In all remaining cases, the inflation tiling is of non-PV
    type, and the diffraction measure, apart from the trivial peak
    at\/ $0$, is purely singular continuous.
\end{enumerate}
\end{theorem}

The article is organised as follows. We begin with the introduction of
our family of inflations and their properties in
Section~\ref{sec:general}, where the cases (1) and (2) of
Theorem~\ref{thm:main-goal} will already follow, and then discuss
the displacement structure and its consequence on the pair
correlations in the form of exact renormalisation relations in
Section~\ref{sec:pair}. This has strong implications on the
autocorrelation and diffraction measures (Section~\ref{sec:auto}),
which are then further analysed via Lyapunov exponents in
Section~\ref{sec:Lyapunov}.  The main result here is the absence of
absolutely continuous diffraction for all members of our family of
inflation systems. One ingredient is the logarithmic Mahler measure of
a derived family of polynomials, which we analyse a little further in
an Appendix.

\section{Setting and general results}\label{sec:general}

Consider the family of primitive substitution rules on the binary
alphabet $\{0,1\}$ given by
\[
   \varrho_{m}^{}: \quad 0\mapsto 01^m\, ,\; 1\mapsto 0\, ,
   \quad
   \text{with } m\in\mathbb{N}. 
\]
Its substitution matrix is
$M_{m}=\left(\begin{smallmatrix}1 & 1\\ m &
    0 \end{smallmatrix}\right)$
with eigenvalues
$\lambda_{m}^{\pm} =\frac{1}{2} \bigl(1\pm\sqrt{4m+1}\,\bigr)$, which
are the roots of $x^2 - x -m=0$. Whenever the context is clear, we
will simply write $\lambda$ instead of $\lambda^{+}_{m}$. Note that
one has $\lambda^{-}_{m} = \frac{-m}{\lambda}$. For each $m\in\NN$,
there is a unique bi-infinite fixed point $w$ of $\varrho^{2}_{m}$
with legal seed $0|0$ around the reference point (or origin), and the
orbit closure of $w$ under the shift action defines the discrete (or
symbolic) hull $\XX_m$. Then, $(\XX_m , \ZZ)$ is a topological
dynamical system that is strictly ergodic by standard results; see
\cite{Q,TAO} and references therein for background and further details.

The Perron--Frobenius (PF) eigenvector of $M_m$, in
frequency-normalised form, is
\begin{equation}\label{eq:freq}
   v^{}_{\mathrm{PF}}\, = \, (\nu^{}_{0} , \nu^{}_{1} )^{T} \, = \,
   \tfrac{1}{\lambda} (1,\lambda \! - \! 1)^{T}  ,
\end{equation}
where the $\nu^{}_{i}$ are the relative frequencies of the two letters
in any element of the hull, $\XX_m$. Next, $(\lambda,1)$ is the
corresponding left eigenvector, which gives the interval lengths for
the corresponding geometric inflation rule.  Up to scale, this is the
unique choice to obtain a self-similar \emph{inflation tiling} of the
line from $\varrho^{}_{m}$; see \cite[Ch.~4]{TAO} for background. 
This version, where $0$ and $1$ stand for intervals of length $\lambda$
and $1$, is convenient because $\ZZ[\lambda]$ is then the natural
$\ZZ$-module to work with. The tiling hull $\YY_{\nts m}$ emerges from
the orbit closure of the tiling defined by $w$, now under the
continuous translation action of $\RR$. The topological dynamical
system $(\YY_{\nts m} , \RR)$ is again strictly ergodic, which can
be proved by a suspension argument \cite{EW}. The unique
invariant probability measure on $\YY_{\nts m}$ is the well-known
patch frequency measure of the inflation rule.

\subsection{Cases with pure point spectrum}

Let us begin with the analysis of the case $m=1$, which defines the
Fibonacci chain. Here, the following result is standard \cite{TAO,Q}.

\begin{fact}\label{fact:Fibo}
  For\/ $m=1$, our substitution defines the well-known Fibonacci chain
  or tiling system. Both dynamical systems, $(\XX^{}_{1}, \ZZ)$ and\/
  $(\YY^{}_{\! 1}, \RR)$, are known to have pure point diffraction and
  dynamical spectrum.  \qed
\end{fact}

Let us thus analyse the systems for $m>1$, where we begin with an easy
observation.

\begin{fact}\label{fact:integer}
  The inflation multiplier\/ $\lambda = \lambda^{+}_{m}$ is an integer
  if and only if\/ $m=\ell(\ell+1)$ with $\ell\in\mathbb{N}$, where\/
  $\lambda^{+}_{m}=\ell+1$ and\/ $\lambda^{-}_{m}=-\ell$.  In all
  remaining cases with $m>1$, the inflation multiplier fails to be a
  PV number.  \qed
\end{fact}

Let us take a closer look at the cases where $\lambda$ is an integer,
where we employ the concept of \emph{mutual local derivability} (MLD)
from \cite{TAO}. This can be viewed as the natural extension of
conjugacy via sliding block maps from symbolic dynamics to tiling
dynamics.

\begin{prop}\label{prop:MLD}
  When\/ $m=\ell(\ell+1)$ with $\ell\in\mathbb{N}$, the inflation
  tiling hull\/ $\YY^{}_{\nts m}$ defined by\/ $\varrho^{}_{m}$ is MLD
  with another inflation tiling hull that is generated by the binary
  constant length substitution\/ $\tilde{\varrho}^{}_{m}$, defined
  by\/ $a \mapsto a b^{\ell}$, $b \mapsto a^{\ell+1}$, under the
  identifications\/ $a\,\widehat{=}\, 0$ and\/
  $b\,\widehat{=}\,1^{\ell+1}$.
  
  Consequently, for any such\/ $m$, the dynamical system\/ 
  $(\YY^{}_{\nts m}, \RR)$ has pure point spectrum, both in
  the dynamical and in the diffraction sense.
  \end{prop}

\begin{proof}
  The claim can be proved by comparing the two-sided fixed point $w$
  of $\varrho^{\ts 2}_{m}$, with seed $0|0$, with that of
  $\tilde{\varrho}^{\ts 2}_{m}$, with matching seed $a|a$, called $u$
  say, where we employ the tiling picture and assume that the letters
  $a$ and $b$ both stand for intervals of length
  $\lambda=\ell+1$. Clearly, the local mapping defined by $a\mapsto 0$
  and $b\mapsto 1^{\ell+1}$ sends $u$ to $w$. For the other direction,
  each $0$ is mapped to $a$, while the symbol $1$ in $w$ occurs in
  blocks of length $\ell (\ell+1)$, which are locally
  recognisable. Any such block is then replaced by $b^\ell$, and this
  defines a local mapping that sends $w$ to $u$. The transfer from the
  symbolic fixed points to the corresponding tilings is consistent, as
  the interval lengths match the geometric constraints.  The extension
  to the entire hulls is standard.
   
  The constant length substitution $\tilde{\varrho}^{}_{m}$ has a
  coincidence in the first position, and thus defines a discrete
  dynamical system with pure point dynamical spectrum by Dekking's
  theorem \cite{Dekking}. Due to the constant length nature,
  $\tilde{\YY}^{}_{\nts m}$ emerges from $\tilde{\XX}^{}_{m}$ by a
  simple suspension with a constant roof function \cite{EW}, so that
  the dynamical spectrum of $(\tilde{\YY}^{}_{\nts m}, \RR)$, and 
  hence that of   $(\YY^{}_{\nts m}, \RR)$ by conjugacy, is still pure
  point.  By the equivalence theorem between dynamical and diffraction
  spectra \cite{LMS,BL} in the pure point case, the last claim is
  clear.
\end{proof}

Let us mention in passing that all eigenfunctions are continuous for
primitive inflation rules \cite{Q,Boris}.  For the systems considered
in this paper, all eigenvalues are thus topological.  So far, we have
the following result.

\begin{theorem}\label{thm1}
  Consider the inflation tiling, with prototiles of natural length,
  defined by\/ $\varrho^{}_{m}$.  For\/ $m=1$ and\/ $m=\ell (\ell+1)$
  with\/ $\ell \in \mathbb{N}$, the tiling has pure point diffraction,
  which can be calculated with the projection method.\footnote{For
  $m\ne 1$, this works analogously to the case of the period doubling 
  sequence; compare  \cite{TAO}.}  The corresponding tiling
  dynamical system\/ $(\YY_{\nts m}, \RR)$ is strictly ergodic and has
  pure point dynamical spectrum.  \qed
\end{theorem}

\subsection{Non-PV cases}

In all remaining cases, meaning those that are not covered by
Theorem~\ref{thm1}, the PF eigenvalue is irrational, but fails to be a
PV number. None of the corresponding tilings can have non-trivial
point spectrum \cite{Boris,BFGR}.  In particular, the only Bragg peak
in the diffraction measure is the trivial one at $k=0$. If we consider
Dirac combs with point measures at the left endpoints of the
intervals, which leads to the Dirac comb of Eq.~\eqref{eq:def-comb}
below, this Bragg peak has intensity
\begin{equation}\label{eq:0-Bragg}
  I^{}_{0} \, = \, \big| \nu^{}_{0} \ts u^{}_{0} + 
   \nu^{}_{1} u^{}_{1} \big|^{2} ,
\end{equation}
where $u^{}_{0}$, $u^{}_{1}$ are the (possibly complex) weights for
the two types of points, and $\nu^{}_{0}$, $\nu^{}_{1}$ are the
frequencies from Eq.~\eqref{eq:freq}; compare \cite[Prop.~9.2]{TAO}.
This gives the first part of the following result, the full proof of which will 
later follow from Lemma~\ref{lem:cond-no-ac} and 
Proposition~\ref{prop:bounded-away}.

\begin{theorem}\label{thm1b}
  For all cases of our inflation family that remain after 
  Theorem~$\ref{thm1}$, the pure point part of the diffraction
  consists of the trivial Bragg peak at\/ $0$, with intensity\/
  $I^{}_{0}$ according to Eq.~\eqref{eq:0-Bragg}, while the remainder
  of the diffraction is purely singular continuous.  
\end{theorem}

The first example in our family with continuous spectral component is
$m=3$, where the eigenvalues are
$\frac{1}{2}\bigl(1 \nts \pm \nts \sqrt{13}\,\bigr)$. This case was
studied in detail in \cite{BFGR}, where also general methods were
developed that can be used for the entire family, as we shall
demonstrate below.

\subsection{Some notation}

To continue, we need various standard results from the theory of
unbounded (but translation bounded) measures on $\RR$, for instance as
summarised in \cite[Ch.~8]{TAO}. In particular, we use $\delta_x$ to
denote the normalised Dirac measure at $x$ and
$\delta^{}_S = \sum_{x\in S}^{} \delta_{x}$ for the Dirac comb of a
discrete point set $S$. A translation bounded measure $\mu$ is simultaneously
considered as a regular Borel measure (then evaluated on bounded Borel
sets) and as a Radon measure (hence as a linear functional on
$C_{\mathsf{c}} (\RR)$, the space of continuous functions with compact
support), which is justified by the general Riesz--Markov
representation theorem; compare \cite[Ch.~4]{S}.

For a continuous function $g$, the measure $g.\mu$ is defined by
$\mu \circ g^{-1}$ as a Borel measure, while $g(.) \mu$ stands for the
measure that is absolutely continuous relative to $\mu$ with
Radon--Nikodym density $g$. The function $\widetilde{g}$ is specified
by $\widetilde{g} (x) = \overline{g(-x)}$, which extends to Radon
measures by $\widetilde{\mu} (g) = \overline{\mu (\widetilde{g} )}$;
for further details, we refer to \cite{TAO}.

\section{Displacement structure and pair 
correlations}\label{sec:pair}

Let $m\in \NN$ be arbitrary, but fixed, which will be suppressed in
our notation from now on whenever reasonable. We will now
review some properties of the inflation structure and how this
can be used to get exact renormalisation relations for the
pair correlation functions.

\subsection{Displacements and their algebraic structure}

First, we quantify the relative displacements of tiles in the
inflation process by the set-valued matrix
\[
     T \, = \, \begin{pmatrix}  \{ 0 \} & \{ 0 \} \\
     S &  \varnothing \end{pmatrix},
     \quad \text{with } S := \{ \lambda, \lambda +  1, 
     \ldots , \lambda + m \! - \! 1\} \ts ,
\]
where $T = (T^{}_{ij})^{}_{0 \leqslant i,j \leqslant 1}$ with
$T^{}_{ij}$ being the set of relative positions of tiles (intervals)
of type $i$ in supertiles of type $j$. Here and below, the positions
are always determined between the left endpoints of the tiles as
markers.

{}From the measure matrix
$\delta^{}_{T} := \bigl(\delta^{}_{T_{ij}} \bigr)_{0 \leqslant i,j
  \leqslant 1}$,
with $\,\widehat{.}\,$ denoting the standard Fourier transform as used
in \cite[Ch.~8]{TAO}, one obtains the Fourier matrix $B$ of our
inflation system as
\begin{equation}\label{eq:F-mat}
    B (k) \, := \, \overline{\widehat{\delta^{}_{T}} (k)} \, = \,
    D^{}_{0} +  p(k) \, D^{}_{\lambda}
\end{equation}
with the trigonometric polynomial
\begin{equation}\label{eq:p-def}
    p(k) \, = \, z^{\lambda} ( 1 + z + \ldots + z^{m-1})
    \big| _{z = \ee^{2 \pi \ii k}}
\end{equation}
and digit matrices $D_0 = \left( \begin{smallmatrix} 1 & 1 \\ 0 & 0
  \end{smallmatrix} \right)$
and $D_\lambda = \left( \begin{smallmatrix} 0 & 0 \\ 1 &
    0 \end{smallmatrix} \right)$,
which are the same matrices for all $m \in \NN$. The complex algebra
$\cB$ generated by them is the \emph{inflation displacement algebra}
(IDA) introduced in \cite{BG15}.  Invoking \cite{BFGR}, one has the
following result.

\begin{fact}\label{fact:IDA}
  For any\/ $m\in\NN$, the IDA\/ $\cB$ of the inflation defined by\/
  $\varrho^{}_{m}$, with intervals of natural length as prototiles, is
  the full matrix algebra, $\Mat(2,\CC)$.  This is also the IDA for
  all powers of the inflation. \qed
\end{fact}

The matrix function defined by $B(k)$ is analytic in $k$, and either
$1$-periodic (whenever $\lambda$ is an integer) or quasiperiodic with
incommensurate base frequencies $1$ and $\lambda$. The case $m=1$ is
somewhat degenerate in this setting, as we shall explain in more
detail later, in Section~\ref{sec:Fibo}.  In the genuinely
quasiperiodic situation, via standard results from the theory of
quasiperiodic functions, one has the representation
\begin{equation}\label{eq:quasi-B}
    B(k) \, = \, \tilde{B} (x,y) 
    \big|_{x=\lambda k, \ts y=k}
\end{equation}
with $\tilde{B} (x,y) = \left( \begin{smallmatrix}
1 & 1 \\ \tilde{p} (x,y) & 0 \end{smallmatrix} \right)$
and 
\begin{equation}\label{eq:quasi-p}
    \tilde{p} (x,y) \, = \, \ee^{2 \pi \ii x}
    (1 + z + \ldots + z^{m-1})\big|_{z = \ee^{2 \pi \ii y}} .
\end{equation}
Here, both $\tilde{p}$ and $\tilde{B}$ are $1$-periodic
in \emph{both} arguments. Our representation is chosen such that
we have the correspondence
\begin{equation}\label{eq:correspond}
   k \mapsto \lambda k  \quad \longleftrightarrow \quad
   (x,y) \mapsto (x,y) M \ts ,
\end{equation}
where $M=M_m$ is the substitution matrix of $\varrho = \varrho^{}_{m}$.
Each such $M$ defines a toral endomorphism on the $2$-torus,
$\TT^2$.

\subsection{Kronecker products}\label{sec:Kronecker}

Below, we also need the matrices
$A(k) = B (k) \otimes \overline{B(k)}$ for $k\in \RR$, which act on
the space $W \! := \CC^2 \otimes^{}_{\CC} \CC^2$, and the
structure of the $\RR$-algebra $\cA$ generated by them. While the
\mbox{$\CC$-algebra} $\cB$ is irreducible by Fact~\ref{fact:IDA},
$\cA$ is not, because each of its elements commutes with the
$\RR$-linear mapping $C \! : \, W \! \xrightarrow{\quad} W$
defined by $x \otimes y \mapsto \overline{y} \otimes \overline{x}$,
where $W$ is considered as an $\RR$-vector space of dimension
$8$. Now, $W = W_{\! +} \oplus W_{\! -}$ with
$W^{}_{\! \pm} := \{ w \in W : C(w) = \pm\ts  w \}$, where
$\dim^{}_{\RR} (W_{\! +}) = \dim^{}_{\RR} (W_{\! -}) = 4$ and
$W_{\! -} = \ii \ts W_{\! +}$. These spaces are invariant under $\cA$, and one
has the following result.

\begin{lemma}\label{lem:real-alg}
  The\/ $\RR$-algebra\/ $\cA$ satisfies\/ $\dim^{}_{\RR} (\cA) = 16$
  and acts irreducibly on each of the four-dimensional invariant
  subspaces\/ $W_{\! +}$ and\/ $W_{\! -}$ from above.
\end{lemma}

\begin{proof}
  The argument is analogous to that in the proof of
  \cite[Lemma~5.3]{BFGR}.  In particular, considering $\Mat (4,\CC)$
  as an $\RR$-algebra of dimension $32$, the subalgebra $\cA$ is
  conjugate to $\Mat (4, \RR) \subset \Mat (4,\CC)$ via the
  conjugation $(.) \mapsto U (.) U^{-1}$ with the unitary matrix
\[
      U \, = \, \myfrac{1}{\sqrt{2}} \begin{pmatrix}
      1\! - \! \ii & 0 & 0 & 0 \\ 0 & 1 & -\ii & 0 \\
      0 & -\ii & 1 & 0 \\ 0 & 0 & 0 & 1\! - \! \ii  \end{pmatrix} ,
\]
where $U (W^{}_{\! \mp}) = \frac{1\pm \ii}{\sqrt{2}}\, \RR^4$.
\end{proof}

One can check that $[ U, A (0)] = 0$. More generally, one has
\[
    A^{}_{U} (k) \, = \, U \nts A (k) \ts U^{-1} \, =  \begin{pmatrix}
    1 & 1 & 1 & 1 \\ c(k)  +  s(k) & s(k) & c(k) & 0 \\
    c(k)  -  s(k) & c(k) & -s(k) & 0 \\ c(k)^2 + s(k)^2 & 0 & 0 & 0
    \end{pmatrix} 
\]
with
\[
  c(k) \, = \sum_{\ell=0}^{m-1} \cos \bigl( 2 \pi (\lambda+\ell) k\bigr)
  \quad \text{and} \quad
  s(k) \, = \sum_{\ell=0}^{m-1} \sin \bigl( 2 \pi (\lambda+\ell) k\bigr).
\]
This gives $c(k)^2 + s(k)^2 = \bigl| p(k) \bigr|^2 = \bigl(
\sum_{\ell=0}^{m-1} \cos \bigl((2\ell+1-m)\pi k \bigr) \bigr)^{\nts 2}$ 
and $A^{}_{U} (0) = A (0) = M \otimes M$ with the substitution
matrix $M=M_m$.

Observe that $A(0)^2$ is a strictly positive matrix, with determinant
$\det (M)^4 > 0$. Now, consider
$A^{(2)}_{U} (k) := A^{}_{U} \bigl( \frac{k}{\lambda}\bigr) A^{}_{U}
(k)$, which defines a smooth, real-valued matrix function with
$\lim_{k\to 0} A^{(2)}_{U} (k) = A (0)^2$. Consequently, there is some
$\varepsilon = \varepsilon (m) > 0$ such that $A^{(2)}_{U} (k)$ is
strictly positive, with positive determinant, for all
$ \lvert k \rvert \leqslant \varepsilon$. We can then state the
following result, the proof of which is identical to that of
\cite[Prop.~5.4]{BFGR}.

\begin{prop}\label{prop:explode}
  Let\/ $k \in [0,\varepsilon]$ with the above choice of\/
  $\varepsilon$, and consider the iteration
\[
     w^{}_n \, := \, A^{(2)}_{U} \bigl( \tfrac{k}{\lambda^{2n-2}}\bigr) 
     \cdots A^{(2)}_{U} \bigl( \tfrac{k}{\lambda^2}\bigr) 
     A^{(2)}_{U} (k) \, w^{}_0
\]    
for\/ $n\geqslant 1$ and any non-negative starting vector\/
$w^{}_0 \ne 0$. Then, the vector\/ $w_n$ will be strictly positive for
all\/ $n \in \NN$ and, as\/ $n \to \infty$, it will diverge with
asymptotic growth\/ $c \ts \lambda^{4n} \ts w^{}_{\mathrm{PF}}$.
Here, $c$ is a constant that depends on\/ $w^{}_0$ and\/ $k$, while\/
$w^{}_{\mathrm{PF}} = v^{}_{\mathrm{PF}} \otimes v^{}_{\mathrm{PF}}$
is the statistically normalised PF eigenvector of\/ $M\otimes M$,
with eigenvalue\/ $\lambda^2$  
and\/ $v^{}_{\mathrm{PF}}$ as in Eq.~\eqref{eq:freq}.  \qed
\end{prop}

As we shall see later, this growth behaviour will collide with
a local integrability condition, and then help to simplify our
spectral problem by a dimensional reduction.

\subsection{Pair correlations}

To introduce the pair correlation functions, let $\YY = \YY_{\nts m}$ be
the tiling hull introduced earlier. Any $\cT \in \YY$ is built from
two prototiles (of length $\lambda>1$ and $1$, respectively).  We now
define the corresponding point set $\vL$ via the left endpoints of the
tiles in $\cT$, so $\vL = \vL^{(0)} \,\dot{\cup}\, \vL^{(1)}$ with
$\vL^{(i)}$ denoting the left endpoints of type $i$. Clearly, $\cT$
and $\vL$ are MLD, as are their hulls; see \cite{TAO} for background. 
By slight abuse of notation, we use $\YY$ for both hulls, which means 
that we implicitly identify these two viewpoints.

Any two elements of $\YY$ are \emph{locally indistinguishable} (LI),
so $\YY$ consists of a single LI class, see \cite[Thm.~4.1]{TAO} or
\cite{Q}, which has the following important consequence.

\begin{fact}\label{fact:constant}
  For any\/ $i,j\in\{0,1\}$, the difference set\/
  $\vL^{(i)} - \vL^{(j)}$ is constant on the hull, which means that it
  does not depend on the choice of $\vL\in\YY$.  \qed
\end{fact}

Given $\vL\in\YY$, let $\nu^{}_{ij} (z)$ denote the (relative)
frequency of occurrence of a point of type $i$ (left) and one of type
$j$ (right) at distance $z$, where
$\nu^{}_{ij} (-z) = \nu^{}_{ji} (z)$.  By the strict ergodicity of our
system, any such frequency exists (and uniformly so), and is the same
for all $\vL\in\YY$. One can write the frequency as a limit,
\[
    \nu^{}_{ij} (z) \, = \lim_{r\to\infty} 
    \frac{\card \bigl( \vL^{(i)}_{r} \cap (\vL^{(j)}_{r} \nts - z) \bigr)}
         {\card (\vL^{}_{r})}
    \, = \, \myfrac{1}{\dens (\vL)} \lim_{r\to\infty}
    \frac{\card \bigl( \vL^{(i)}_{r} \cap (\vL^{(j)}_{r}
       \nts - z ) \bigr)}{2r} \ts ,
\]
where $\vL\in\YY$, with $\vL = \vL^{(0)} \, \dot{\cup} \, \vL^{(1)}$
as above. The lower index $r$ indicates the intersection of a set with
the interval $[-r,r]$.  Moreover, one has
\[
   \nu^{}_{ij} (z) \, > \, 0  \quad
   \Longleftrightarrow \quad
   z \in \varDelta^{}_{ij} := \vL^{(j)} - \vL^{(i)}  \ts ,
\]
and
$\vU^{}_{\nts\nts ij} := \sum_{z \in \varDelta_{ij}} \nu^{}_{ij} (z)
\, \delta^{}_{z}$
defines a positive pure point measure on $\RR$ with locally finite
support. Note that $\vU^{}_{ii}$ is also positive definite. We call
the $\nu^{}_{ij} (z)$ the \emph{pair correlation coefficients} and the
$\vU^{}_{\nts\nts ij}$ the corresponding \emph{pair correlation
measures} of $\YY$, where
$\vU^{}_{\nts \nts ij} (\{ z \}) = \nu^{}_{ij} (z)$. Our relative
normalisation means that we have
$\nu^{}_{00} (0) + \nu^{}_{11} (0) = 1$, so that
\[
    \nu^{}_{00} (0) \, = \, \nu^{}_{0} 
    \quad \text{and} \quad
    \nu^{}_{11} (0) \, = \, \nu^{}_{1}
\]
are the relative tile (or letter) frequencies from
Eq.~\eqref{eq:freq}.

\begin{prop}\label{prop:pair-reno}
  The pair correlation coefficients of\/ $\YY$ satisfy the exact
  renormalisation relations
\[
   \nu^{}_{ij} (z) \, = \, \myfrac{1}{\lambda}
   \sum_{k,\ell} \sum_{r\in T_{ik}} \sum_{s\in T_{j\ell}}
   \nu^{}_{k\ell} \Bigl( \myfrac{z+r-s}{\lambda} \Bigr)
\]
for any\/ $i,j\in\{0,1\}$, subject to the condition that\/
$\nu^{}_{mn} (z) = 0$ whenever\/ $z \notin \varDelta^{}_{mn}$. 
In terms of the measures\/ $\vU^{}_{\nts\nts ij}$, this
amounts to the convolution identity
\[
    \vU \, = \, \myfrac{1}{\lambda}
    \Bigl( \widetilde{\delta^{}_{T}} \overset{*}{\otimes}
    \delta^{}_{T} \Bigr) * (f \nts \nts . \ts \vU) \ts ,
\]
where\/ $f$ is the dilation defined by\/ $x \mapsto \lambda x$ and\/
$\overset{*}{\otimes}$ denotes the Kronecker convolution product, 
while\/ $\vU$ stands for the measure vector\/
$(\vU^{}_{00}, \vU^{}_{01}, \vU^{}_{\nts 10}, \vU^{}_{\nts 11})^{T}$.
\end{prop}

\begin{proof}[Sketch of proof]
  For $m=1$, this was shown in \cite{BG15}, while the case $m=3$ is
  treated in \cite{BFGR}. In general, the underlying observation is
  that, due the aperiodicity of $\varrho$ and the ensuing local
  recognisability, each tile lies in a unique \mbox{level{\ts}-$1$}
  supertile, and the frequencies for the distance between tiles can
  uniquely be related to the frequencies of (generally different)
  supertile distances, after a change of scale. A simple computation 
  then gives the first relation (a more general version of which will 
  appear in \cite{BGM}).

  The second identity, in the form of measures, follows from the
  definition of the $\vU^{}_{\nts \nts ij}$ by a straightforward
  calculation; see \cite{BFGR} for details in the case $m=3$.
\end{proof}

\section{Autocorrelation and diffraction}\label{sec:auto}

As above, $m\in\NN$ is arbitrary but fixed. For $\vL\in\YY$, we
consider the weighted Dirac comb
\begin{equation}\label{eq:def-comb}
    \omega^{}_{u} \, = \sum_{x\in\vL} u^{}_{x} \, \delta^{}_{x}
   \, = \, u^{}_{0} \, \delta^{}_{\! \!\vL^{(0)}_{\vphantom{a}}} + u^{}_{1} \,
   \delta^{}_{\! \! \vL^{(1)}_{\vphantom{a}} }\ts ,
\end{equation}
where $u^{}_{0}$ and $u^{}_{1}$ are the (possibly complex-valued)
weights of the two types of points. The corresponding
\emph{autocorrelation measure}, or autocorrelation for short,
is defined by the volume-\-averaged (or Eberlein) convolution
\[
    \gamma^{}_{u} \, = \, \omega^{}_{u} \circledast\ts
   \widetilde{\omega^{}_{u}} \, = \, \dens (\vL)
   \sum_{ij} \overline{u^{}_{i}} \: \vU^{}_{\nts \nts ij}
   \ts\ts u^{}_{j} \ts ,
\]
where we refer to \cite[Chs.~8 and 9]{TAO} for the general setting and
to \cite{BFGR} for the detailed calculations in the case $m=3$.
Existence and uniqueness (with independence of $\vL$) are again a
consequence of the strict (and hence in particular unique) ergodicity
of our system.

Since all the measures
$\vU^{}_{\nts \nts ij} = (\delta^{}_{\nts -\vL^{(i)}} \nts \circledast
\delta^{}_{\! \vL^{(j)}} )/\dens (\vL)$
are well-defined Eberlein convolutions of translation
bounded measures and hence Fourier transformable by
\cite[Lemma~1]{BG15}, we also have the relation
\begin{equation}\label{eq:q-form}
    \widehat{\gamma^{}_{u}} \, = \, \dens (\vL)
    \sum_{ij} \overline{u^{}_{i}} \: \widehat{\vU}^{}_{\nts ij}
    \ts\ts u^{}_{j}
\end{equation}
after Fourier transform, where $\widehat{\gamma^{}_{u}}$ is
a positive measure, for any $u \in \CC^2$. Note that each
$\widehat{\vU}^{}_{\nts ij}$ is a positive definite measure
on $\RR$,  and also positive for $i=j$.
Since $\widehat{\vU}^{}_{\nts ij} = \widehat{\vU^{}_{\nts ij}}$ 
by definition, one has
\[
       \overline{\widehat{\vU^{}_{\nts ij}}}     \, = \,
       \widehat{\widetilde{\vU^{}_{\nts ij}}}    
        \, = \,   \widehat{\vU^{}_{\nts \nts ji}} \ts .
\]
This, in combination with Eq.~\eqref{eq:q-form}, implies the
following property.

\begin{fact}\label{fact:Hermitian}
  For any bounded Borel set\/ $\cE \subset \RR$, the complex matrix\/
  $\bigl( \widehat{\vU}^{}_{\nts ij} (\cE)\bigr)_{0 \leqslant i,j
    \leqslant 1}$ is Hermitian and positive semi-definite.  \qed
\end{fact}

Note that, since $\widehat{\vU}^{}_{00}$ and
$\widehat{\vU}^{}_{\nts 11}$ are positive measures on $\RR$, the
positive semi-definiteness of the matrix
$\bigl( \widehat{\vU}^{}_{\nts ij} (\cE)\bigr)$ is simply equivalent
to the determinant condition
$\det \bigl( \widehat{\vU}^{}_{\nts ij} (\cE)\bigr) \geqslant
0$. \smallskip

As a counterpart to Proposition~\ref{prop:pair-reno}, with
$\widehat{\vU}$ denoting the vector of Fourier transforms of the pair
correlation measures, we get the following result.

\begin{prop}\label{prop:F-reno}
  Under Fourier transform, the second identity of
  Proposition~$\ref{prop:pair-reno}$ turns into the relation
\[
   \widehat{\vU} \, = \, \myfrac{1}{\lambda^2}\,
   A (.) \cdot \bigl( f^{-1} \! . \ts \widehat{\vU}\bigr) ,
\]
where\/ $A(k) = B(k) \otimes \overline{B(k)}$ with the Fourier
matrix\/ $B(k)$ from Eq.~\eqref{eq:F-mat} and\/ $f(x) = \lambda x$.
\qed
\end{prop}

\subsection{Pure point part}

By \cite[Lemma~6.1]{TAO}, the identity from
Proposition~\ref{prop:F-reno} must hold for each spectral component of
$\widehat{\vU}$ separately. We write the pure point part as
\[
    \bigl( \widehat{\vU} \bigr)_{\mathsf{pp}} \, =
    \sum_{k\in K} \cI (k) \, \delta^{}_{k}
\]
with the intensity vector $\cI (k) = \widehat{\vU} (\{ k \})$ and $K$
the support of the pure point part, which is (at most) a countable
set. Without loss of generality, we may assume that
$\lambda K \subseteq K$, possibly after enlarging $K$ appropriately.
Inserting this into the above identity, one obtains
\begin{equation}\label{eq:pp-reno}
   \cI (k) \, = \, \myfrac{1}{\lambda^2} \ts\ts
       A(k) \, \cI(\lambda k) \ts .
\end{equation}
In particular, this gives $A(0) \, \cI(0) = \lambda^2 \, \cI (0)$,
which means
\begin{equation}\label{eq:central}
  \widehat{\vU}^{}_{\nts \nts ij} (\{ 0 \}) \, = \,
  \cI (0) \, = \, \nu^{}_{i} \, \nu^{}_{j} \, = \,
  \frac{\dens(\vL^{(i)}) \dens(\vL^{(j)})}{\bigl(\dens(\vL)\bigr)^2}
\end{equation}
with the relative frequencies $\nu^{}_i$ from Eq.~\eqref{eq:freq}.

\subsection{Conditions on absolutely continuous part}

Likewise, if we represent $(\widehat{\vU})^{}_{\mathsf{ac}}$ by the
vector $h$ of its Radon--Nikodym densities relative to
Lebesgue measure, one obtains \cite{BFGR}
\begin{equation}\label{eq:dens-reno}
   h(k) \, = \, \myfrac{1}{\lambda} \,
   A(k) \, h(\lambda k) \ts ,
\end{equation}
which holds for a.e.\ $k\in\RR$. Here, the different exponent for the
prefactor in comparison to Eq.~\eqref{eq:pp-reno} is the crucial point
to observe and harvest. Since $\det (B(k)) = - p(k) = 0$ holds if and
only if $k\in Z_m := \frac{1}{m} \ZZ \setminus \ZZ$, the matrix $A(k)$
is invertible for all $k \notin Z_m$, hence for a.e.\ $k\in\RR$.  For
such $k$, we also have
\begin{equation}\label{eq:A-outward}
    h (\lambda k) \, = \, \lambda \ts\ts A^{-1} (k) \, h (k) \ts ,
\end{equation}
which is the outward-going counterpart to Eq.~\eqref{eq:dens-reno}.

If we interpret the vector $h$ as a matrix
$(h^{}_{ij})^{}_{0 \leqslant i,j \leqslant 1}$, the iterations from
Eqs.~\eqref{eq:dens-reno} and \eqref{eq:A-outward} can also be written
as
\begin{equation}\label{eq:iter-1}
\begin{split}
     \bigl( h^{}_{ij} \bigl(\tfrac{k}{\lambda})\bigr) \, & = \,
     \lambda^{-1} B \bigl(\tfrac{k}{\lambda}\bigr) 
     \bigl( h^{}_{ij} (k) \bigr)
     B^{\dagger} \bigl(\tfrac{k}{\lambda}\bigr)
     \quad \text{and} \quad \\[2mm]
     \bigl( h^{}_{ij} (\lambda k)\bigr) \, & = \,
     \lambda B^{-1} (k) \bigl( h^{}_{ij} (k) \bigr)
     ( B^{\dagger})^{-1} (k) \ts ,
\end{split}
\end{equation}
where $B^{\dagger}$ denotes the Hermitian adjoint of $B$.
This suggests a suitable decomposition of $h$.

\begin{lemma}\label{lem:decompose}
  For a.e.\ $k\in\RR$, the Radon--Nikodym matrix\/
  $\bigl( h^{}_{ij} (k)\bigr)$ is Hermitian and positive
  semi-definite. Moreover, it is of rank at most\/ $1$.
\end{lemma}

\begin{proof}
  From Fact~\ref{fact:Hermitian}, we know that, given any bounded
  Borel set $\cE$, the complex matrix
  $\bigl( \widehat{\vU}^{}_{ij} (\cE) \bigr)$ is Hermitian and
  positive semi-definite, which also holds for the absolutely
  continuous part of $\widehat{\vU}$.  By standard arguments, this
  implies the first claim.

For a.e.\ $k\in\RR$, the Radon--Nikodym matrix is then of the
form $H=\left( \begin{smallmatrix} a & b+\ii c \\ b - \ii c & d
  \end{smallmatrix} \right)$
with $a, b, c, d \in \RR$, $a, d \geqslant 0$ and
$ad \geqslant (b^2 + c^2)\geqslant 0$. When $\det (H) = 0$,
the rank of $H$ is at most $1$.
On the other hand, when $\det (H) >0$, the rank of $H$ is $2$,
with $ad > 0$.  In general, we have a unique decomposition as
\begin{equation}\label{eq:H-split}
     H \, = \, \begin{pmatrix} a' & b+\ii c \\ 
        b - \ii c & d \end{pmatrix}
     + a'' \begin{pmatrix} 1 & 0 \\ 0 & 0 \end{pmatrix}
\end{equation}
with $ a = a' + a''$ such that each matrix on the right-hand side
is positive semi-definite (hence $a' \geqslant 0$ and
$a''\geqslant 0$) and of rank at most $1$ (which means
$a'd = b^2 + c^2 \ts $).

It suffices to prove our second claim for a.e.\
$k\in \bigl[ \frac{\varepsilon}{\lambda}, \varepsilon \bigr]$ for some
$\varepsilon > 0$, as the two iterations in Eq.~\eqref{eq:iter-1}
transport the property to all $k>0$, and then to $k<0$ via
$h^{}_{ij} (-k) = h^{}_{ji} (k)$. Here, we choose $\varepsilon$ as in
Proposition~\ref{prop:explode}. Whenever $h^{}_{11} (k) = 0$, we have
$\det \bigl( h^{}_{ij} (k) \bigr) = 0$ and the Radon--Nikodym matrix
has rank at most $1$. Otherwise, we define
$h^{\prime \prime}_{00} (k) = \det\bigl( h^{}_{ij} (k) \bigr)/
h^{}_{11} (k)$
and
$h^{\prime}_{00} (k) = h^{}_{00} (k) - h^{\prime \prime}_{00} (k)$,
which are measurable functions and achieve the decomposition explained
previously.

To continue, we switch to the vector notation from
Eqs.~\eqref{eq:dens-reno} and \eqref{eq:A-outward}, and observe that
the matrix
$\left( \begin{smallmatrix} 1 & 0 \\ 0 & 0 \end{smallmatrix} \right)$
corresponds to the vector $w^{}_{0} = (1,0,0,0)^T \in W_{\! +}$ in the
notation of Section~\ref{sec:Kronecker}. Now,
$U \ts w^{}_{0} = \frac{1-\ii}{\sqrt{2}}\, w^{}_{0}$, and the
iteration of $w^{}_{0} $ under $A^{}_{U} (k)$,
\[
     w^{}_{n} \, = \, A^{}_{U} \bigl( \tfrac{k}{\lambda^{n-1}} 
     \bigr) \cdots A^{}_{U} \bigl( \tfrac{k}{\lambda} \bigr) 
     \ts A^{}_{U} (k) \, w^{}_{0} \ts ,
\] 
grows asymptotically as $c \lambda^{2n} w^{}_{\mathrm{PF}}$ by an
application of Proposition~\ref{prop:explode}, where $c>0$ depends on
$k$. Observing that
$U^{-1} \ts w^{}_{\mathrm{PF}} = \frac{1+\ii}{\sqrt{2}} \,
w^{}_{\mathrm{PF}}$
and applying a standard argument on the basis of Lusin's theorem as in
the proof of \cite[Lemma~6.5]{BFGR}, we see that
$h^{\prime \prime}_{00} (k) > 0$ would behave proportional to
$k^{-1}$ as $k \, {\scriptstyle \searrow} \, 0$, which is impossible
for a locally integrable function.  Since also
$h^{\prime}_{00} (k) \geqslant 0$, there cannot be any cancellation
with the other term of our decomposition under the inward iteration,
and we must conclude that $h^{\prime \prime}_{00} (k) =0$ for a.e.\
$k \in \bigl[ \frac{\varepsilon}{\lambda}, \varepsilon \bigr]$, and
hence for a.e.\ $k\in\RR$ as argued above.  This implies our claim.
\end{proof}

\subsection{Dimensional reduction and Lyapunov exponents}

If $H\in \Mat (2,\CC)$ is Hermitian, positive semi-definite and of
rank at most $1$, there are two complex numbers, $v^{}_{0}$
and $v^{}_{1}$ say, such that $H^{}_{ij} = v^{}_{i} \,
\overline{v^{}_{j}}$ holds for $i,j \in \{ 0,1\}$; compare
\cite[Fact~6.6]{BFGR} for more.
The main consequence of Lemma~\ref{lem:decompose} now is
that it suffices to consider a vector  $v(k) =
\bigl( v^{}_{0} (k) , v^{}_{1} (k) \bigr)^T$ of functions from
$L^{2}_{\mathrm{loc}} (\RR)$ under the simpler iterations
\[
   v \bigl( \tfrac{k}{\lambda} \bigr) \, = \,
    \myfrac{1}{\sqrt{\lambda}} \, B \bigl( \tfrac{k}{\lambda}
    \bigr) \ts v (k)
   \quad \text{and} \quad
   v(\lambda k) \, = \, \sqrt{\lambda} \, B^{-1} (k) \, v (k) \ts ,
\]
the latter for $k \notin Z_m$. In particular, one has
\begin{equation}\label{eq:outward-iter}
   v(\lambda^n k) \, = \, \lambda^{n/2} \,
   B^{-1} (\lambda^{n-1} k) \cdots B^{-1} (k) \, v (k) \ts ,
\end{equation}
which holds for a.e.\ $k\in \RR \setminus \bigcup_{\ell=0}^{n-1}
\lambda^{-\ell} Z_m$, and thus still for a.e.\ $k\in\RR$.

There are at most two Lyapunov exponents for this iteration,
which agree with the extremal exponents \cite{Viana} defined
by
\begin{equation}\label{eq:extremal}
\begin{split}
   \chi^{}_{\max} (k) \, & = \, \log \sqrt{\lambda} 
    \, + \,
    \limsup_{n\to\infty} \myfrac{1}{n} \log \bigl\|
    B^{-1} (\lambda^{n-1} k) \cdots B^{-1} (\lambda k)
    \ts B^{-1} (k) \bigr\|
    \quad \text{and} \\[2mm]
   \chi^{}_{\min} (k) \, & = \, \log\sqrt{\lambda} 
    \, - \,
    \limsup_{n\to\infty} \myfrac{1}{n} \log \bigl\|
    B(k) \ts B(\lambda k) \cdots B(\lambda^{n-1} k) \bigr\| ,
\end{split}
\end{equation}
where the term $\log\sqrt{\lambda}$ emerges from the prefactor on the
right-hand side of Eq.~\eqref{eq:outward-iter}. Our main concern
will be the minimal exponent, for the following reason \cite{BFGR,BGM}.

\begin{lemma}\label{lem:cond-no-ac}
  If\/ $\chi^{}_{\min} (k) \geqslant c > 0$ holds for a.e.\ $k$ in a
  small interval, the diffraction measure\/ $\widehat{\gamma^{}_{u}}$
  is a singular measure, for any non-trivial choice of the weights\/
  $u^{}_{0}$ and\/ $u^{}_{1}$, which means\/
  $u^{}_{1}, u^{}_{2} \in \CC$ with\/ $u^{}_{1} u^{}_{2} \ne 0$.  \qed
\end{lemma}

\section{Analysis via Lyapunov exponents}\label{sec:Lyapunov}

Let $m\in\NN$ be fixed, and $\lambda = \lambda^{+}_{m}$ as before.
Consider the matrix cocycle
\[
   B^{(n)} (k) \, := \, B(k) \ts B(\lambda k) 
   \cdots B(\lambda^{n-1} k) \ts ,
\] 
which is motivated by Eq.~\eqref{eq:extremal}. Note that $B^{(n)} (k)$
is invertible for
$k \notin \bigcup_{\ell=0}^{n-1} \lambda^{-\ell} Z_m$.  When $m=1$, one
has $Z^{}_{1} = \varnothing$ and
$\bigl| \det \bigl(B(k)\bigr) \bigr| \equiv 1$, which makes this case
considerably simpler. In general, one has the following result.

\begin{prop}\label{prop:det-lim}
   For a.e.\ $k\in\RR$, one has 
    $ \,\lim_{n\to\infty} \,  \frac{1}{n} \log \ts \bigl| 
    \det \bigl(B^{(n)} (k) \bigr) \bigr|  =  0$.
\end{prop}

\begin{proof}
  For $m=1$, one has
  $\bigl| \det \bigl( B^{(n)} (k) \bigr) \bigr| \equiv 1$, and the
  claim trivially holds for all $k\in\RR$.  When $m\geqslant 2$, we
  invoke Sobol's theorem, as outlined in \cite{BFGR}; see \cite{BHL}
  for a detailed exposition. Clearly, $\det\bigl( B (k))$ is a Bohr
  almost periodic function, but it has zeros for $k\in
  Z_m$.
  Consequently, $\log \ts \bigl| \det\bigl( B(k)\bigr) \bigr|$ cannot
  be Bohr almost periodic, because it has singularities at these
  points. Nevertheless, this function is locally Lebesgue-integrable
  on $\RR$, and is continuous on $\RR\setminus Z_m$, hence locally
  Riemann-integrable on the complement of
  $Z_m + (-\varepsilon, \varepsilon)$ for any $\varepsilon>0$.  Then,
  Sobol's theorem \cite{Sobol} (in the periodic case where $\lambda$
  is an integer) or its extension to almost periodic functions
  \cite{BHL} can be applied as follows.

  First, recall that the sequence
  $(\lambda^n k)^{}_{n\in\NN}$ is uniformly distributed modulo $1$ for
  a.e.\ $k\in\RR$. Next, one needs the  property that
  $Z_m$ is a Delone set and that, for any fixed $\varepsilon > 0$,
  and then for a.e.\ $k\in\RR$, the inequality
\[
   \mathrm{dist} (\lambda^{n-1} k, Z_m) \, \geqslant \, 
   \myfrac{1}{n^{1+\varepsilon}}
\]
holds for almost all $n\in\NN$ (meaning for all except at most
finitely many); see \cite[Lemma~6.2.6]{BHL}. Then, again for any fixed
$\varepsilon > 0$, it follows from \cite[Thm.~5.13]{Harman} that the
discrepancy of $(\lambda^n k)^{}_{n\in\NN}$, for a.e.\ $k\in\RR$, is
\[
     \cD^{}_{N} \, = \, \cO \left(
     \frac{\bigl(\log(N)\bigr)^{\frac{3}{2} + 
    \varepsilon}}{\sqrt{N}} \right)
    \quad \text{as } N \to \infty \ts .
\]

Putting this together, we can apply \cite[Thm.~6.4.8]{BHL} which tells
us that the Birkhoff-type averages of the function
$\log \ts \bigl| \det \bigl( B(.) \bigr) \bigr|$, for a.e.\ $k\in\RR$,
converge to the mean of this function (see Eq.~\eqref{eq:def-mean} 
below for a definition), which gives
\[
\begin{split}
    \myfrac{1}{n}  \log \ts & \bigl| 
    \det \bigl( B^{(n)} (k) \bigr) \bigr|  \,  =  \,
    \myfrac{1}{n} \sum_{\ell=0}^{n-1}   
    \log \ts \bigl| \det \bigl( B ( \lambda^{\ell} k) \bigr)
    \bigr|  \\[2mm]
    & \xrightarrow{\, n\to\infty\,} \int_{0}^{1}\!
    \log \ts \bigl| \det \bigl( B(t) \bigr) \bigr| \dd t
    \, =  \int_{0}^{1} \! \log \ts \bigl| 1 + z + \ldots + z^{m-1} 
    \bigr|_{z=\ee^{2 \pi \ii t}} \dd t   \; = \; 0 \ts ,
\end{split}
\]
where the last step follows via Jensen's formula from complex analysis
(see \cite[Prop.~16.1]{K-book} for a formulation that fits our
situation) because the polynomial $1+z+\ldots +z^{m-1}$ either equals
$1$ (when $m=1$) or has zeros only on the unit circle.
\end{proof}

\begin{remark}
  When $m=\ell (\ell+1)$ with $\ell\in\NN$, where
  $\lambda = \ell + 1$, the result of Proposition~\ref{prop:det-lim}
  easily follows from Birkhoff's ergodic theorem, because
  $\det \bigl( B(k)\bigr)$ is then a $1$-periodic, locally
  Lebesgue-integrable function that gets averaged along orbits of the
  dynamical system defined by $x \mapsto \lambda x \bmod{1}$ on the
  $1$-torus, $\TT$. This approach, however, does not extend to the
  other values of $m$ with $m>1$, because the sequence
  $(\lambda^n k)^{}_{n\in\NN}$, taken modulo $1$, is then no longer an
  orbit of $x \mapsto \lambda x \bmod{1}$ on $\TT$; compare
  \cite[Ex.~6.3.4]{BHL}.  \exend
\end{remark}

We can now relate the two extremal exponents from
Eq.~\eqref{eq:extremal} as follows.

\begin{lemma}\label{lem:expo-sum}
  For a.e.\ $k\in\RR$, one has\/ $\,
  \chi^{}_{\max} (k) + \chi^{}_{\min} 
  (k) = \log(\lambda)$.
\end{lemma}

\begin{proof}
  Recall that, for any invertible matrix $B$, one has
  $B^{-1} = \frac{1}{\det (B)} \, B^{\mathsf{ad}}$, where
  $B^{\mathsf{ad}}$ is the (classical) adjoint of $B$. 
  The adjoint satisfies $(A B)^{\mathsf{ad}} = B^{\mathsf{ad}} 
  A^{\mathsf{ad}}$.

  Now, in the formula for the extremal exponents, we are free to
  choose any matrix norm, as this does not affect the limit.  For
  $2\! \times \! 2$-matrices, one has
  $\| B^{\mathsf{ad}} \|^{}_{\mathrm{F}} = \| B \|^{}_{\mathrm{F}}$,
  where $\|.\|^{}_{\mathrm{F}}$ denotes the Frobenius norm. Our claim
  now follows from Proposition~\ref{prop:det-lim} after a simple
  calculation.
\end{proof}

In view of Lemma~\ref{lem:expo-sum}, we define
\[
    \chi^{B} (k) \, := \, \limsup_{n\to\infty} \myfrac{1}{n} \log
    \bigl\| B(k) \ts B(\lambda k) \cdots B(\lambda^{n-1} k)
    \bigr\| ,
\]
so that
$\chi^{}_{\max} (k) = \log\sqrt{\lambda} + \chi^{B} (k)$ and
$\chi^{}_{\min} (k) = \log\sqrt{\lambda} - \chi^{B} (k)$
holds for a.e.\ $k\in\RR$, together with $\chi^{}_{\max} (k) \geqslant
\chi^{}_{\min} (k)$. We can thus simply analyse $\chi^{B}$ from now
on, which clearly is a non-negative function. \smallskip

\subsection{Arguments in common}\label{sec:common}

Below, we need the \emph{mean} of a function. If $f$ is a Bohr almost
periodic function on $\RR$ (and thus in particular uniformly
continuous and bounded), its mean, $\MM (f)$, is defined by
\begin{equation}\label{eq:def-mean}
   \MM (f) \, = \lim_{T\to\infty} \myfrac{1}{T}
   \int_{x}^{x+T} \! \! f(t) \dd t \ts ,
\end{equation}
where $x\in\RR$ is arbitrary. By standard results, the mean of such an
$f$ exists for all $x\in\RR$, is independent of $x$, and the
convergence is actually uniform in $x$. This is also true when $f$ is
almost periodic in the sense of Stepanov, which in particular covers
some of our later situations; see \cite{Cord,BHL} for details. When
$f$ is a periodic function, with fundamental period $T$, the mean is
simply given by $\MM (f) = \frac{1}{T} \int_{0}^{T}\! f(t) \dd t$.

Observing that $\lvert p(.)\rvert^2$ with $p$ from
Eq.~\eqref{eq:p-def} is $1$-periodic (while $p$ itself need not be),
the simplest sufficient criterion for the positivity of all Lyapunov
exponents is given by
\begin{equation}\label{eq:m1}
     \log (\lambda) \, > \, \MM \bigl( \log
      \| B(.) \|^{2}_{\mathrm{F}}\bigr)  \, = 
     \int_{0}^{1} \! \log \bigl( 2 + \lvert p (t) 
      \rvert^{2} \bigr) \dd t
     \, = \int_{0}^{1} \! \log \big\lvert q (z)
      \big\rvert_{z=\ee^{2 \pi \ii t}} \dd t
     \, = \, \fm (q)
\end{equation}
where $q$ is the polynomial 
\begin{equation}\label{eq:q-poly}
   q (z) \, = \, 2 z^{m-1} + \bigl(1+z+\ldots+z^{m-1}\bigr)^{2}
\end{equation}
and the validity of $\overline{z} = z^{-1}$ on the unit circle was
used. Here, $\fm (q)$ denotes the \emph{logarithmic Mahler measure} of
$q$; see \cite{EW2,K-book} for background.  The integral can now once
again be calculated by means of Jensen's formula; see the Appendix for
some details.  The comparison between $\log (\lambda)$ and $\fm (q)$
is illustrated in Figure~\ref{fig:mlp}. \smallskip

\begin{figure}
\includegraphics[width=0.8\textwidth]{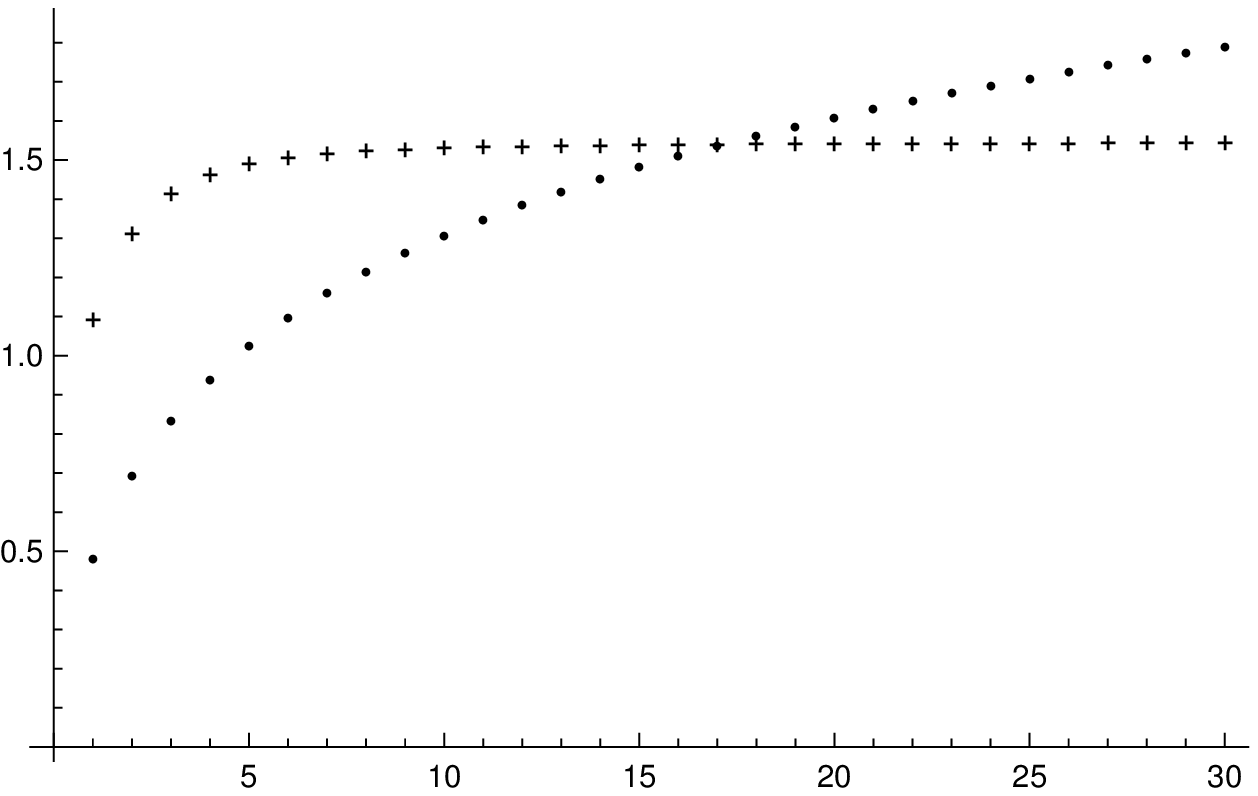}
\caption{The values of $\log(\lambda)$ (dots) and 
$\fm (q)$ (crosses) for $1\leqslant m\leqslant 30$.}
\label{fig:mlp}
\end{figure}

More generally, one has the following result.

\begin{lemma}\label{lem:18}
  For any\/ $m \geqslant 18$ and then a.e.\ $k\in\RR$, all Lyapunov
  exponents of the outward iteration \eqref{eq:outward-iter} are
  strictly positive and bounded away from\/ $0$.
\end{lemma}

\begin{proof}
  Since $\chi^{}_{\max} (k) \geqslant \chi^{}_{\min} (k)$,
  we need to show that $\log\sqrt{\lambda} -c \geqslant \chi^{B} (k)$ 
  holds for some $c>0$ and a.e.\ $k\in \RR$.  A sufficient criterion for this is the
  inequality from Eq.~\eqref{eq:m1}. Since $\fm (q)$ is bounded, see
  Lemma~\ref{lem:m-bound} from the Appendix, it is clear that this
  inequality holds for all sufficiently large $m\in\NN$. By
  Lemma~\ref{lem:m-bound}, this is so for all $m\geqslant 40$, and the
  slightly better estimate from Remark~\ref{rem:better} improves this
  to all $m\geqslant 23$.

  In any case, a (precise) numerical investigation of the remaining
  cases shows that that our claim is indeed true for all
  $m \geqslant 18$; compare Figure~\ref{fig:mlp}.
\end{proof}

In order to establish our goal for the remaining values of $m$, we need 
to determine a suitable $N = N(m)$ such that
\begin{equation}\label{eq:compare}
     \log (\lambda) \, > \, \myfrac{1}{N} \,\MM \bigl( \log
     \| B^{(N)} (.) \|^{2}_{\mathrm{F}} \bigr) \, = \,
     \begin{cases}
     N^{-1} \int_{0}^{1} \log 
     \big\| B^{(N)} (k) \big\|^{2}_{\mathrm{F}} \dd k \ts , &
     \text{if $\lambda\in\ZZ$}, \\[1mm]
     N^{-1}
     \int_{[0,1]^2} \log \big\| \tilde{B}^{(N)} (x,y) \big\|^{2}_{\mathrm{F}}
     \dd x \dd y \ts , & \text{otherwise}. \end{cases}
\end{equation}
When $\lambda$ is not an integer, $\| B^{(N)} (.)\|^{2}_{\mathrm{F}}$
is generally not a periodic, but a quasiperiodic function. In this
case, we use the representation as a section through a doubly
$1$-periodic function according to Eq.~\eqref{eq:quasi-B}, which
permits the simple expression for the mean in \eqref{eq:compare}. The
latter can now be calculated numerically with good precision, and
without ambiguity. Note that the choice of the Frobenius norm
$\|.\|^{}_{\mathrm{F}}$ does not give the best bounds, but is rather
convenient otherwise.  The result is given in Table~\ref{tab:num},
with minimal values for $N(m)$. Consequently, we can sharpen
Lemma~\ref{lem:18} and complete the proof of Theorem~\ref{thm1b}
as follows.

\begin{prop}\label{prop:bounded-away}
  For any\/ $m\in\NN$ and then a.e.\ $k\in\RR$, all Lyapunov
  exponents of the outward iteration \eqref{eq:outward-iter} are
  strictly positive and bounded away from\/ $0$.  \qed
\end{prop}

\begin{table}
\label{tab:num}
\caption{Some relevant values for the quantities in the inequality of 
  Eq.~\eqref{eq:compare}. 
  The numerical error is less than $10^{-3}$ in all cases listed.}
\renewcommand{\arraystretch}{1.2}
\begin{tabular}{|c|c@{\;\,\;}c@{\;\,\;}c@{\;\,\;}c@{\;\,\;}c
      @{\;\,\;}c@{\;\,\;}c@{\;\,\;}c@{\;\,\;}c@{\;\,\;}c|}
\hline
$m$ & 1 & 2 & 3 & 4 & 5 & 6 & 7 & 8 & 9 & 10 \\
\hline
$\log(\lambda)$ & 
0.481 & 0.693 & 0.834 & 0.941 & 1.027 & 1.099 & 1.161 &
1.216 & 1.265 & 1.309 \\
$N=N(m)$ & 
 6 & 4 & 4 & 3 & 3 & 3 & 2 & 2 & 2 & 2 \\
$\tfrac{1}{N} \MM \bigl( \log \| B^{(N)}(.) \|^{2}_{\mathrm{F}} \bigr)$ & 
 0.439 & 0.677 & 0.770 & 0.924 & 0.949 & 0.964 & 1.144 & 1.152 & 
 1.157 & 1.161 \\[0.5mm]
\hline\hline
$m$ & 11 & 12 & 13 & 14 & 15 & 16 & 17 & 18 & 19 & 20 \\
\hline
$\log(\lambda)$ & 
1.349 & 1.386 & 1.421 & 1.453 & 1.483 & 1.511 & 1.538 & 
1.563 & 1.587 & 1.609\\
$N=N(m)$ & 
2 & 2 & 2 & 2 & 2 & 2 & 2 & 1 & 1 & 1 \\
$\tfrac{1}{N} \MM \bigl( \log \| B^{(N)}(.) \|^{2}_{\mathrm{F}} \bigr)$ & 
   1.164 & 1.166 & 1.168 & 1.169 & 1.170 & 1.171 & 1.172
 & 1.546 & 1.547 & 1.547 \\[0.5mm]
\hline
\end{tabular}
\end{table}

\subsection{The Fibonacci case}\label{sec:Fibo}

Here, the leading eigenvalue is $\lambda = \tau$, the golden ratio,
which is a PV number. Essentially as a consequence of \cite[Thm.~2.9
and Prop.~3.8]{FSS}, which need some modification and extension to be
applicable here, the extremal Lyapunov exponents exist as limits, for
a.e.\ $k \in \RR$. Let us look into this in more detail, in a slightly
different way that provides an independent derivation of this
property.  Here, we have
\[
     B (k) \, = \, \begin{pmatrix} 1 & 1 \\
     \ee^{2 \pi \ii \tau k} & 0 \end{pmatrix} ,
\]
which is $\tau^{-1}$-periodic. However, this observation does 
not help because already
\[
    B^{(2)} (k) \, = \, B (k) \, B (\tau k) \, = \,
    \begin{pmatrix} 1 + \ee^{2 \pi \ii (\tau + 1) k} & 1 \\
    \ee^{2 \pi i \tau k} & \ee^{2 \pi i \tau k} \end{pmatrix}
\]
is genuinely quasiperiodic, with fundamental frequencies $\tau$ and
$1$. In line with our general approach from Eqs.~\eqref{eq:quasi-B} 
and \eqref{eq:quasi-p}, we now define
$\tilde{B}^{(n+1)} (x,y) = \tilde{B} (x,y) \, \tilde{B}^{(n)} \bigl(
(x,y) M \bigr)$
with 
\[
   \tilde{B}^{(1)} (x,y) \, = \, \tilde{B} (x,y) \, = \,
   \begin{pmatrix} 1 & 1 \\ \ee^{2 \pi \ii x} & 0 \end{pmatrix}
   \quad \text{and} \quad
   M \, = \, \begin{pmatrix} 1 & 1 \\ 1 & 0 \end{pmatrix}.
\]
Then, $\tilde{B}^{(n)} (x,y)$ defines a matrix cocycle over the
dynamical system defined on $\TT^{2}$ by the toral automorphism
$(x,y) \mapsto (x,y) M \bmod{1}$. By Oseledec's theorem, see
\cite{Viana}, the Lyapunov exponents for $\tilde{B}^{(n)}$ exist as
limits, for a.e.\ $(x,y)\in\TT^2$, and are constant.

However, what we really need is the existence of the Lyapunov
exponents for
\[
   B^{(n)} (k) \, = \, \tilde{B}^{(n)} (x,y)
       |^{}_{x = \tau k , \, y = k}
\]
for a.e.\ $k\in\RR$, which is a statement along the line
$\RR (\tau, 1)$, respectively its wrap-up on $\TT^2$. This is the
subspace defined by the left PF eigenvector of $M$. The problem here
is that this defines a null set for Lebesgue measure on $\TT^2$, so
that the previous argument does not immediately imply what we
need. However, for Lebesgue-a.e.\ starting point on the line
$\RR (\tau, 1)$, the iteration sequence on this line, taken modulo
$1$, is also equidistributed in $\TT^2$, by standard arguments around
Weyl's lemma. This allows for a relation between the result on the
line and that on $\TT^2$ as follows.

Any initial condition for the cocycle $\tilde{B}^{(n)}$ is
following an orbit of the toral automorphism that converges,
exponentially fast, towards an orbit on this special subspace.  This
is a consequence of the PV property of $\tau$ and the fact that the
second eigenvalue of $M$ is $1-\tau \approx - 0.618$, the algebraic
conjugate of $\tau$. Assume that the Lyapunov exponents for $B^{(n)}$
fail to exist as limits for a subset of $\RR (\tau, 1)$ of positive
measure. Then, this must also be true of the exponents for
$\tilde{B}^{(n)}$ for all initial conditions that lead to orbits which
approach the failing orbits on $\RR (\tau, 1)$. By standard arguments,
these initial conditions would constitute a set of positive measure,
now with respect to Lebesgue measure on $\TT^2$, in contradiction to
our previous finding. We thus have the following result.

\begin{fact}
  For\/ $m=1$ and a.e.\ $k\in\RR$, the Lyapunov exponents from
  Eq.~\eqref{eq:extremal} exist as limits, and are constant. \qed
\end{fact}

One can check numerically that $\chi^{B} (k) \approx 0.16\ts\ts (3)$
in this case, and some further analysis with a Furstenberg-type
representation should result in a more reliable value.

\subsection{Integer inflation multipliers}\label{sec:integer}

In these cases, we know the absence of any continuous spectral
components already from Proposition~\ref{prop:MLD}.  Moreover, 
in view of Lemma~\ref{lem:cond-no-ac}, our
treatment in Section~\ref{sec:common} also 
confirms the absence of absolutely continuous
diffraction via the Lyapunov exponents.  Here, the exponents also
exist as limits for a.e.\ $k\in \RR$, by an application of Oseledec's
theorem to the matrix cocycle, viewed over the dynamical system
defined on $\TT$ by $x \mapsto \lambda x \bmod{1}$ with
$\lambda = \ell + 1$ according to Fact~\ref{fact:integer}.

\begin{fact}
  When\/ $m = \ell (\ell+1)$ for\/ $\ell\in\NN$, hence\/
  $\lambda = \ell+1$, the Lyapunov exponents from
  Eq.~\eqref{eq:extremal} exist as limits, for a.e.\ $k\in\RR$, and
  are constant.   \qed
\end{fact}

As another way to look at the problem,
let us add a quick analysis of the constant length substitution
\begin{equation}\label{eq:alt-int}
   \tilde{\varrho}^{}_{m} : \, \begin{cases}
   a \mapsto a b^{\ell} &  \\ b \mapsto a^{\ell+1} &
   \end{cases}
\end{equation}
with $\ell\in\NN$, which defines a hull that is MLD to the one defined
via $\varrho^{}_{m}$ for $m=\ell(\ell+1)$ by Proposition~\ref{prop:MLD}, so
the spectral type of both systems must be the same.  The displacement
matrix is
\[
    T \, = \, \begin{pmatrix} 0 & \{ 0, 1, 2, \ldots , \ell\} \\
   \{ 1,2, \ldots , \ell\} & \varnothing \end{pmatrix} ,
\]
which results in the Fourier matrix
\[
    B(k) \, = \, \begin{pmatrix} 1 &  \psi^{}_{\ell} (z) \\
    z \, \psi^{}_{\ell-1} (z) & 0 \end{pmatrix}_{z=\ee^{2 \pi \ii k}}
\]
with $\psi^{}_{\ell} (z) := 1 + z + \ldots + z^{\ell}$.
One gets an analogue to Eq.~\eqref{eq:m1} in the form
\[
    \MM \bigl( \log  \| B(.) \|^{2}_{\mathrm{F}}\bigr)  \, = 
    \int_{0}^{1} \log\ts \biggl| \frac{s(z)}{(z-1)^2}
    \biggr|^{}_{z=\ee^{2 \pi \ii t}} \dd t
    \, = \, \fm (s) \ts ,
\]
with $s(z) = z^{2\ell+2} + z^{2\ell+1} + z^{\ell+2} - 6 z^{\ell+1}
+z^{\ell} + z + 1$.
As we explain in more detail in the Appendix, we used
$\fm \bigl( (z-1)^2 \bigr) = 0$ in an intermediate step.

One could now repeat the general analysis of Section~\ref{sec:common}
in this case, with an outcome of a similar kind. However, there is a
more efficient way as follows. First, observe that we now have
$\det ( B(k)) = - z \, \psi^{}_{\ell-1} (z) \, \psi^{}_{\ell} (z)$
with $z=\ee^{2 \pi \ii k}$, which is a product of a monic polynomial
(in the variable $z$) with two cyclotomic ones.  Consequently, the
corresponding logarithmic Mahler measures vanish, and we once
again get
\begin{equation}\label{eq:det-2}
   \lim_{n\to\infty} \myfrac{1}{n} \log \ts\ts \bigl|
   \det (B^{(n)} (k) ) \bigr| \, =  \int_{0}^{1}
   \log \ts\ts \bigl| \det ( B(t)) \bigr| \dd t \, = \,
   \fm \bigl(z \, \psi^{}_{\ell-1} (z) \, \psi^{}_{\ell} 
   (z) \bigr) \, = \, 0 \ts ,
\end{equation}
for a.e.\ $k\in \RR$, as in Proposition~\ref{prop:det-lim}. 
Next, observe that $v = (1,1)$ is a common left eigenvector
of $B (k)$ for all $k\in\RR$, with eigenvalue $\psi^{}_{\ell} (z)$
for $z=\ee^{2 \pi \ii k}$. This gives
\[
   \lim_{n\to\infty} \myfrac{1}{n} \log \ts\ts
   \bigl\| v \ts B^{(n)} (k) \bigr\| \, = \, \fm (\psi^{}_{\ell})
   \, = \, 0
\]
for a.e.\ $k\in\RR$. In view of Eq.~\eqref{eq:det-2}, this implies
that both exponents of the cocycle $B^{(n)}$ vanish in this case.

Now, we still have $\chi^{}_{\min} + \chi^{}_{\max} = \log (\lambda)$
as in Lemma~\ref{lem:expo-sum}, where $\lambda = {\ell+1}$,
despite the fact that we now consider the substitution from
Eq.~\eqref{eq:alt-int}.  With this derivation, we have actually shown
the following result.

\begin{coro}
  The extremal Lyapunov exponents for the outward iteration defined by
  the constant-length substitution \eqref{eq:alt-int} are equal, and
  given by\/
  $\, \chi^{}_{\min} = \chi^{}_{\max} =  \log \sqrt{\ell + 1}$.
  \qed
\end{coro}

Also this approach implies the diffraction spectrum
to be singular. However, as before, this is only a consistency check
because Dekking's criterion (see the proof of Proposition~\ref{prop:MLD})
already gives a stronger result, namely the pure point nature of the spectrum.

\section*{Appendix}

Here, we consider some logarithmic Mahler measures, in particular
$\fm (q)$ for the polynomial $q$ from Eq.~\eqref{eq:q-poly} with
$m\in\NN$. The polynomials $q$ seem to be irreducible over $\ZZ$,
though we have no general proof for this observation.  As follows from
a simple calculation, $\fm (q)$ takes the values $\log (3)$ for $m=1$
and $\log \bigl( 2 + \sqrt{3}\, \bigr)$ for $m=2$. It is known that
one must have $\fm (q) = \log (\xi)$ where $\xi$ is a Perron number.
A little experimentation shows that $\xi$ is a Salem number for $m=3$,
namely the largest root of $z^4 - 3 z^3 - 4 z^2 - 3 z + 1$, and a
Pisot number for $m=4$, this time the largest root of
$z^4 - 4 z^3 - 2 z^2 + 2 z + 1$. For $m=5$, one finds that $\xi$ is
the largest root of $z^8 - 6 z^7 + 7 z^6 - 3 z^4 + 7 z^2 - 6 z + 1$,
which is genuinely Perron, as the second largest root of this
irreducible polynomial, with approximate value $1.354 > 1$, lies
outside the unit circle. It would be interesting to know more 
about the numbers that show up here.

More generally, expressing $\lvert p (t)\rvert^2$ in Eq.~\eqref{eq:m1}
as $\bigl(\sin (m \pi t)/\sin (\pi t)\bigr)^2$, one has
\begin{equation}\label{eq:sin-form}
   \fm (q) \, = 
   \int_{0}^{1} \log \left( 2 + \left(
   \frac{\sin (m \pi t)}{\sin (\pi t)} \right)^{\! 2}
   \, \right) \dd t \ts .
\end{equation}
Since $\sin(m \pi t)^2 \leqslant 1$, one gets a simple
upper bound as
\[
\begin{split}
   \fm (q) \, & \leqslant \int_{0}^{1} \log
   \frac{1+2\, \sin(\pi t)^2}{\sin (\pi t)^2} \dd t
   \, = \, \log (2) + \int_{0}^{1} \log 
   \frac{2 - \cos (2 \pi t)}{1 - \cos (2 \pi t)} 
   \dd t \\[2mm]
   &  = \, \log (2) + \fm \bigl(z^2 - 4 z + 1 \bigr) - 
     \fm \bigl( (z-1)^2 \bigr) \, = \,
     \log \bigl( 4 + 2 \sqrt{3} \, \bigr)
     \, \approx \, 2.010 \ts ,
\end{split}
\]
where the logarithmic Mahler measures of the quadratic polynomials
were evaluated via Jensen's formula again.  This shows that $\fm (q)$
is bounded for our family of polynomials.  A slightly better bound can
be obtained as follows.

\begin{lemma}\label{lem:m-bound}
   For any\/ $m\in\NN$, the logarithmic Mahler measure of the
   polynomial\/ $q$ from Eq.~\eqref{eq:q-poly} satisfies
   the inequality\/
   $\, \fm (q)  <  \log \sqrt{46}  \approx
     1.914 {\ts\ts} 321$.
\end{lemma}

\begin{proof}
Here, we employ an argument from \cite{Clunie,Borwein} that
was also used, in a similar context, in \cite{Neil}. By a simple
geometric series calculation, one finds that
$q (z) = \frac{r (z)}{(z-1)^2}$ with
\begin{equation}\label{eq:r-def}
   r (z) \, = \, z^{2m} + 2 z^{m+1} - 6 z^{m} + 
   2 z^{m-1} + 1 \, = \sum_{\ell=0}^{2m} c^{}_{\ell} \, z^{\ell}.
\end{equation}
Consequently, we have $\fm (q) = \fm (r) - \fm \bigl( (z-1)^2\bigr)
= \fm (r)$.

Let $\fM (r)  = \exp ( \fm  (r))$ be the (ordinary)  Mahler measure of
$r$;  compare \cite[Sec.~1.2]{EW2}.   By the  strict convexity  of the
exponential function  and Jensen's inequality,  see \cite[Ch.~2.2]{LL}
for a suitable formulation, one finds
\[
    \fM (r) \, < \int_{0}^{1} \bigl| r (z) \bigr|_{z=\ee^{2 \pi \ii t}} \dd t
    \, = \, \| r \|^{}_{1} \, \leqslant \, \| r \|^{}_{2} \ts ,
\]
where $r = r(t)$ is considered as a trigonometric polynomial on $\TT$
(with the usual $1$-periodic extension to $\RR$). In fact, since $r$
is not a monomial, we also have $\| r \|^{}_{1} < \| r \|^{}_{2}$.

Assume that $m\geqslant 2$, so that the exponents of $r (z)$ in
Eq.~\eqref{eq:r-def} are distinct. Consequently, by Parseval's
equation, we may conclude that
\[
     \| r \|^{2}_{2} \, = \sum_{\ell=0}^{2m} \lvert c^{}_{\ell}\rvert^2
     \, = \, 46 \ts ,
\]
so that $\fM (r) < \sqrt{46}$, independently of $m$. This inequality
trivially also holds for $m=1$, and we get
$\fm (q) = \fm (r) < \log \sqrt{46}$ for all $m\in\NN$ as claimed.
\end{proof}

With this bound, one has $\log (\lambda) > \fm (q)$ for all
$m\geqslant 40$, where $\lambda = \lambda^{+}_{m}$ as before.

\begin{remark}\label{rem:better}
  An even better bound can be obtained from Eq.~\eqref{eq:sin-form} by
  observing that, as $t$ varies a little, $\sin (m \pi t)^2$
  oscillates quickly when $m$ is large (with mean $\frac{1}{2}$),
  while $\bigl(\sin (\pi t)\bigr)^2$ remains roughly constant.  Under
  the integral, one can then replace
  $\bigl(\sin (m \pi t)/\sin (\pi t)\bigr)^2$ by
  $\frac{1}{2} \bigl(\sin (\pi t)\bigr)^{-2}$, which still gives an
  upper bound for $\fm (q)$ because
  $\frac{\dd^2}{\dd t^2} \log (t) < 0$ on $\RR_{+}$. Now,
\[
    \fm (q)  \, \leqslant \int_{0}^{1} \log
    \frac{3 - 2 \cos (2 \pi t)}{1 - \cos (2 \pi t)} \dd t
    \, = \, \fm \bigl(z^2 - 3 z + 1 \bigr) + \log (2) \, = \,
    \log \bigl( 3 + \sqrt{5}\, \bigr) 
    \, \approx \, 1.655 {\ts\ts} 571 \ts ,
\]
which is smaller than $\log (\lambda)$, where
$\lambda = \lambda^{+}_{m}$ as above, for all $m \geqslant 23$.
\exend
\end{remark}

The values $\fm (q)$, as a function of $m\in\NN$, seem to be increasing,
so that $\lim_{m\to\infty} \fm (q)$ would be the optimal upper bound.
The limit exists because $\fm (q) = \fm (r)$, and the polynomial $r$
satisfies $r (z) = \tilde{r} (z, z^m)$ with
\[
     \tilde{r}  (z,w) \, = \, -w \left(6 - 2 \bigl( z + 
      z^{-1} \bigr) - \bigl( w + w^{-1} \bigr) \right) .
\]
By a classic approximation theorem for two-dimensional Mahler
measures, see \cite[Thm.~3.21]{EW2}, one has
$\lim_{m\to\infty} \fm \bigl( \tilde{r} (z, z^m) \bigr) = \fm \bigl(
\tilde{r} (z,w)\bigr)$, where
\[
\begin{split}
    \fm (\tilde{r} ) \, & = \int_{\TT^2} \log \bigl( 6 -
    2 \cos (2 \pi t_1) - 4 \cos (2 \pi t_2) \bigr)
    \dd t_1 \dd t_2 \\[1mm]
    & = \, 2 \int_{0}^{1} \mathrm{arsinh} \bigl( \sqrt{2} \,
     \sin (\pi t_2) \bigr) \dd t_2
    \, \approx \, 1.550 {\ts} 675 \ts .
\end{split}
\]
So, when $\fm (q)$ is an increasing function (which we did not prove),
we immediately get the estimate $\log (\lambda) > \fm (q) $ for all
$m\geqslant 18$.

\begin{remark}
  The polynomial $s $ from Section~\ref{sec:integer} can be analysed
  in a completely analogous way. Here, one has
  $s (z) = - z^{\ell +1} \left( 6 - \bigl( z + z^{-1} \bigr) - \bigl(
    w + w^{-1}\bigr) - \bigl( zw + (zw)^{-1} \bigr) \right)$,
  and the approximation theorem results in
\[
    \lim_{\ell\to\infty} \fm (s) \, = \int_{\TT^2}
    \log \bigl( 6 - 2 \cos (2 \pi t_1) - 2 \cos (2 \pi t_2)
     - 2 \cos (2 \pi (t_1 + t_2)) \bigr) \dd t_1 \dd t_2
    \, \approx \, 1.615 \ts .
\]  
  Moreover, various other properties are similar to those of
  the polynomial $q$ from above.
\exend
\end{remark}

\section*{Acknowledgements}

It is our pleasure to thank Michael Coons, David Damanik, Natalie
P.~Frank, Franz G\"{a}hler, Andrew Hubery, E.\ Arthur (Robbie)
Robinson and Boris Solomyak for discussions. 
This work was supported by the German
Research Foundation (DFG), within the CRC 1283.

\end{document}